\documentclass{amsart}

\usepackage[inactive]{srcltx} 
\usepackage{amsmath, amsthm, amscd, amsfonts, amssymb, graphicx, color}
 \newtheorem{thm}{Theorem}[section]
 \newtheorem{cor}[thm]{Corollary}
 \newtheorem{lem}[thm]{Lemma}
 
 \theoremstyle{definition}
 
 \theoremstyle{remark}
 \newtheorem{rem}[thm]{Remark}
\newcommand{\A}{A_{sa}}

\begin{document}

\title[Continuous bilinear maps on Banach $\star$-algebras]
 {Continuous bilinear maps on Banach $\star$-algebras}

\author{ B. Fadaee}
\thanks{{\scriptsize
\hskip -0.4 true cm \emph{MSC(2020)}:  46K05, 47B48, 15A86.
}}

\address{Department of Mathematics, Faculty of Science, University of Kurdistan, P.O. Box 416, Sanandaj, Kurdistan, Iran.}

\email{behroozfadaee@yahoo.com}

\address{}

\email{}

\thanks{}

\thanks{}



\date{}

\dedicatory{}

\commby{}


\begin{abstract}
Let $A$ be a unital Banach $\star $-algebra with unity $1$, $X$ be a Banach space and $ \phi : A \times A \to X  $ be a continuous bilinear map. We characterize the structure of $\phi$ where it satisfies any of the following properties:
\[ a,b \in A, \,\,\, a b^\star = z \, \,(a^\star b=z)\Rightarrow \phi ( a , b^\star ) = \phi ( z, 1 ) \, \, (\phi ( a^\star  , b) = \phi ( z, 1 ));\] 
\[ a,b \in A, \,\,\, a b^\star = z \, \, (a^\star b=z)\Rightarrow \phi ( a , b^\star ) = \phi ( 1, z ) \, \, (\phi ( a^\star  , b) = \phi ( 1, z )),\] 
where $z\in A$ is fixed. 
\end{abstract}

\maketitle
\section{Introduction}
In recent years, several authors studied the linear (additive) maps that behave like homomorphisms, derivations or right (left) centalizers when acting on special products (for instance, see \cite{Al1, barar, Br, che, fad0, fad, fad2, fos, gh5, gh6} and the references therein). The above questions and the question of characterizing linear maps that preserve special products on algebras can be solved by considering bilinear maps that preserve certain product  properties. Motivated by these reasons, Bre\v{s}ar et al. \cite{Br2} introduced the concept of zero product (resp., Jordan product, Lie product) determined algebras. In the continuation of this discussion, the problem of characterizing bilinear maps at specific products was considered. We refer the reader to \cite{Al2, gh1, gh2, gh3, gh4, gh44, wa} and references therein for results concerning characterizing bilinear maps through special products. With the idea of the above, in this article we will characterize continuous bilinear maps on Banach $\star$-algebras through special products based on the action of the involution. Our results can be useful in studying the structure of Banach $\star$-algebras. Proving our main result is also technical and it is based on complex analysis. In the second section, some preliminaires and necessary tools are presented. The third section contains the main results of the article.
\section{Preliminaires}
Let $A$ be a Banach $\star$-algebra. In this article, we will consider the following sets, which are defined based on specific multiplications.
\[ S_A^{r \star} (z) = \{ (a , b ) \in A \times A : a b^\star = z \}, \]
\[ S_A^{l \star} (z) = \{ (a , b ) \in A \times A : a^\star b = z \}, \]
where $z\in A$ is a fixed point. 
\par 
In order to prove our results we need the following lemmas from the complex analysis, see \cite{Silver}.
\begin{lem} \label{lem1}
If series $ \sum_{n=0}^\infty a_n z^n $ converges for all real values of $ z $, then it converges for all $ z \in \mathbb{C} $.
\end{lem}

\begin{lem} \label{lem2}
Suppose that 
\begin{enumerate}
\item[(a)]
a function $ f $ is analytic throughout a domain $ D $;
\item[(b)]
$ f(z) = 0  $ at each point $ z $ of a domain or line segment contained in $ D $.
\end{enumerate}
Then $ f(z) \equiv 0 $ in $ D $; that is, $ f(z) $ is identically equal to zero throughout $ D $.
\end{lem}

\section{Continuous bilinear maps of Banach $\star $-algebras through special products}
In this section we will give our main results. Through this section $A$ is a unital Banach $\star $-algebra with unity $1$, $z\in A$ is a fixed, and $X$ is a Banach space.
\begin{thm} \label{t1}
Let $ X $ be a Banach space and let $ \phi : A \times A \to X  $ be a continuous bilinear map with the property that 
\[   \phi ( a , b^\star ) = \phi ( z, 1 ) ~ for~ all ~ (a, b ) \in S_A^{r\star} (z) ~[ ~ (a, b ) \in S_A^{l \star} (z)] . \]
Then
\[ \phi ( za, a ) = \phi (z a^2  , 1 ) ~ and ~ \phi ( z a, 1) = \phi ( z, a ) , ~~~ a \in A \]
and there exists a continuous linear map $ \Phi : A \to X  $ such that 
\[ \phi ( za,  b) + \phi ( zb ,a ) = \Phi ( a \circ  b ) , ~~~ a, b \in A . \]
\end{thm}
\begin{proof}
Let  $ a \in \A $ and $ t \in \mathbb{R} $. Since $ ( z \exp ( ta) , \exp (- ta) ) \in S_A^{r \star} (z) $, we deduce that
\begin{align*}
\phi ( z, 1 ) & = \phi (z \exp ( t a) , (\exp (-t a))^\star ) \\
& = \phi \left(  z \exp ( t a) , \left( \sum_{m=0}^\infty \dfrac{(-t)^m a^m}{m !} \right)^\star \right) \\
& = \phi \left( \exp ( t a) , \left(   \sum_{m=0}^\infty \dfrac{(- 1)^m t^m  a^m}{m !} \right) \right)  \\
& = \sum_{m=0}^\infty \dfrac{(- 1)^m t^m }{m !} \phi (z \exp ( t a) , a^m ) \\
& = \sum_{m=0}^\infty \dfrac{(- 1)^m i^m }{m !} \phi \left( \sum_{n=0}^\infty \dfrac{ t^n  z a^n}{n !} , a^m \right) \\
& = \sum_{m=0}^\infty\sum_{n=0}^\infty \dfrac{(- 1)^m t^{m+n} }{m ! n!} \phi ( z a^n , a^m) \\
& =  \phi ( z, 1 ) + \sum_{k=1}^\infty t^k \left( \sum_{m+n = k } \dfrac{(- 1)^m}{m ! n!} \phi ( z a^n , a^m) \right) ,
\end{align*}
since $ \phi $ is a continuous bilinear map. Therefore
\begin{equation}\label{eq1}
\sum_{k=1}^\infty t^k \left( \sum_{m+n = k } \dfrac{(- 1)^m}{m ! n!} \phi ( z a^n , a^m) \right) = 0 
\end{equation}  
 for any $ t \in \mathbb{R} $.
\par
 Let $ \tau \in A^* $. We define a map $ f : \mathbb{C} \to \mathbb{C} $ by
\[ f (\lambda) = \tau \left(  \sum_{k=1}^\infty \lambda^k \left( \sum_{m+n = k } \dfrac{(- 1)^m}{m ! n!} \phi ( z a^n , a^m) \right) \right)  \]
for all $ \lambda \in \mathbb{C} $. Hence from \eqref{eq1} for $ t \in \mathbb{R} $, we find that $ f(t) =0 $. So $ f $ is an analytic function on real axis and hence on $ \mathbb{C} $ by Lemma \ref{lem1}. Now since  $ f $ is zero for each point on a real axis, by Lemma \ref{lem2}, $f(\lambda) $ is identically equal to zero throughout $ \mathbb{C} $ and $ \tau $ arbitrary, consequently,
\begin{equation} \label{f1}
\sum_{m+n=k}^{\infty} \dfrac{(-1)^{ m}}{m! n!}  \phi ( z a^n , a^m) =0
\end{equation}
for all $ a \in \A $ and $ k \in \mathbb{N} $. Let $ k=1$, we find  that $ \phi ( z a , 1 ) - \phi ( z, a) = 0  $ and hence 
\begin{equation} \label{f2}
 \phi ( z a, 1) = \phi ( z, a )
\end{equation}
for all $ a \in \A $. Now taking $ k=2 $ in \eqref{f1}, we obtain $ \frac{1}{2} \phi ( za^2, 1 ) - \phi ( za, a ) + \frac{1}{2} \phi ( z, a^2 ) = 0 $ for any $ a \in \A $. So by \eqref{f2} we have 
\begin{equation} \label{f3}
\phi ( z a, a ) = \phi ( z a^2 , 1 ) , ~~~ \forall a \in \A.
\end{equation}
For any $ a, b \in \A $, replacing $ a $ by $ a + b  $ in \eqref{f3}, we get that
\[ \phi ( za, b ) + \phi ( zb ,a ) = \phi ( z (a b + ba ) , 1 ) . \]
If we deone the linear map $ \Phi : A \to X $ by $ \Phi (a) = \phi (za, 1) $, then $ \Phi $ is continuous and 
\[ \phi ( za, b ) + \phi ( zb ,a ) = \Phi ( a \circ b ) \]
for all $ a, b \in \A $. Since every element $ a \in A $ is a combination some self adjoints, we also deduce above statements are for each arbitrary elements of $ A $.
\end{proof}
Let $ a \in \A $ and $ t \in \mathbb{R} $. From $ ( \exp ( ta) z^\star , \exp (-ta) ) \in S_A^{l \star} (z) $, and using similar arguments as proof of above theorem we get the following result.
\begin{thm} \label{t1b}
Let $ X $ be a Banach space and let $ \phi : A \times A \to X  $ be a continuous bilinear map with the property that 
\[   \phi ( a^\star, b ) = \phi ( z, 1 ) ~ for~ all  ~ (a, b ) \in S_A^{l \star} (z). \]
Then
\[ \phi ( za, a ) = \phi (z a^2  , 1 ) ~ and ~ \phi ( z a, 1) = \phi ( z, a ) , ~~~ a \in A \]
and there exists a continuous linear map $ \Phi : A \to X  $ such that 
\[ \phi ( za,  b) + \phi ( zb ,a ) = \Phi ( a \circ  b ) , ~~~ a, b \in A . \]
\end{thm}
\begin{rem} \label{t2}
Let  $ a \in \A $ and $ t \in \mathbb{R} $. By the fact that $ ( \exp ( ta) , z^\star \exp (-ta) ) \in S_A^{r \star} (z) $ and $ ( \exp ( ta) , \exp (-ta) z ) \in S_A^{l \star} (z) $, and using similar arguments as proof of Theorem \ref{t1} we get the following:
\par 
Let $ X $ be a Banach space and let $ \phi : A \times A \to X  $ be a continuous bilinear map. If $ \phi $ satisfies any of the following conditions
\begin{itemize}
\item[(i)] $ \phi ( a , b^\star ) = \phi ( 1, z ) ~ for~ all ~ (a, b ) \in S_A^{r \star} (z);$
\item[(ii)] $ \phi ( a^\star  , b) = \phi ( 1, z ) ~ for~ all ~ (a, b ) \in S_A^{r \star} (z)$,
\end{itemize}
then 
\[ \phi ( a, a z ) = \phi ( 1, a^2 z ) ~ and ~ \phi ( 1 , a z) = \phi ( a, z ) , ~~~ a \in A \]
and there exists a continuous linear map $ \Phi : A \to X  $ such that 
\[ \phi ( a,  b z) + \phi ( b , a  z) = \Phi ( a \circ  b ) , ~~~ a, b \in A . \]
\end{rem}
The results obtained are especially important in the case where the $z\neq 0$. We have the following corollaries.
\begin{cor} \label{c2}
Let $ X $ be a Banach space and let $ \phi : A \times A \to X $ be a continuous
bilinear map. If $ \phi $ satisfies any of the following conditions;
\begin{enumerate}
\item[(i)]
$ a, b \in A, a b^\star = 1 \Rightarrow \phi ( a , b^\star ) = \phi ( 1 , 1)  $, 
\item[(ii)]
$ a , b \in A, a^{\star} b = 1  \Rightarrow \phi ( a^\star , b ) = \phi ( 1 , 1 ) $,
\end{enumerate}
then 
\[ \phi ( a , b ) + \phi ( b , a ) = \phi ( a \circ b , 1) \]
for all $ a ,b \in A $.
\end{cor}
\begin{proof}
The result is clear from Theorems \ref{t1} and \ref{t1b} by letting $ z=1 $. 
\end{proof}
Recall that a bilinear map $ \phi : A \times A \to X $ is called symmetric if $ \phi (a, b) = \phi (b, a) $ holds for all $ a , b \in A $. By Theorem \ref{t1} and Corollary \ref{c2}, the following corollary is obvious.
\begin{cor} \label{c3}
Let $ X $ be a Banach space and let $ \phi : A \times A \to X $ be a continuous
symmetric bilinear map. Then the following conditions are equivalent:
\begin{enumerate}
\item[(i)]
$ a, b \in A, a b^\star = 1 \Rightarrow \phi ( a , b^\star ) = \phi ( 1 , 1)  $;
\item[(ii)]
$ a , b \in A, a^{\star} b = 1  \Rightarrow \phi ( a^\star , b ) = \phi ( 1 , 1 ) $.
\end{enumerate}
\end{cor}

\subsection*{Acknowledgment}
The authors thanks the referees for careful reading of the manuscript and for helpful suggestions.

\bibliographystyle{amsplain}
\bibliography{xbib}

\begin{thebibliography}{20}

\bibitem{Al1} 
 J. Alaminos, M. Bre\v{s}ar, J. Extremera and A. R. Villena, Maps preserving zero products, Studia Math. 193 (2009), 131-159.

\bibitem{Al2}  
 J. Alaminos, M. Bre\v{s}ar, J. Extremera, A.R. Villena, On bilinear maps determined by rank one idempotents. Linear Algebra Appl. 432 (2010), 738--743.
 
 \bibitem{barar}
A. Barari, B. Fadaee and H. Ghahramani, \textit{Linear maps on standard operator algebras characterized by action on zero products}, Bull. Iran. Math. Soc. 45 (2019), 1573--1583.

\bibitem{Br} 
 M. Bre\v{s}ar, Characterizing homomorphisms, multipliers and derivations in rings with idempotents, Proc. R. Soc. Edinb. Sect. A. 137 (2007), 9--21.
 
\bibitem{Br2} 
 M. Bre\v{s}ar, M. Gra\v{s}i\v{c}, and J.S. Ortega, Zero product determined matrix algebras, Linear Algebra Appl. 430 (2009), 1486--1498. 
 
 \bibitem{che}
M.A. Chebotar, W.-F. Ke and P.-H. Lee, \textit{Maps characterized by action on zero products}. Pacific J. Math, 216(2) (2004), 217--228. 

\bibitem{fad0}
B. Fadaee and H. Ghahramani, \textit{Jordan left derivations at the idempotent elements on reflexive algebras }, Publ. Math. Debrecen, 92/3-4 (2018), 261--275. 

\bibitem{fad}
B. Fadaee and H. Ghahramani, \textit{Linear maps on $C^*$-algebras behaving like (Anti-)derivations at orthogonal elements}, Bull. Malays. Math. Sci. Soc. 43 (2020), 2851--2859.

\bibitem{fad2}
B. Fadaee, K. Fallahi and H. Ghahramani, \textit{Characterization of linear mappings on (Banach)$\star$-algebras by similar properties to derivations}, Math. Slovaca, 70(4) (2020), 1003--1011.

\bibitem{fos}
A. Fo\v{s}ner and H. Ghahramani, \textit{Ternary derivations of nest algebras}, Operator and Matrices, 15 (2021), 327--339.

\bibitem{gh1}
H. Ghahramani, Zero product determined some nest algebras, Linear Algebra Appl. 438 (2013) 303--314.
 
\bibitem{gh2}
H. Ghahramani, Zero product determined triangular algebras, Linear Multilinear Algebra 61 (2013) 741--757. 

\bibitem{gh3} H. Ghahramani, \textit{On rings determined by zero products}, J. Algebra and appl. 12 (2013), 1--15.

\bibitem{gh4} H. Ghahramani, \textit{On derivations and Jordan derivations through zero products}, Operator and Matrices, 4 (2014), 759--771. 

\bibitem{gh44} H. Ghahramani, \textit{On centralizers of Banach algebras}, Bull. Malays. Math. Sci. Soc. 38 (2015), 155--164.

\bibitem{gh5}
 H. Ghahramani, \textit{Linear maps on group algebras determined by the action of the derivations or anti-derivations on a set of orthogonal elements}, Results in Mathematics, 73 (2018), 132--146.
 
\bibitem{gh6}
 H. Ghahramani and Z. Pan, \textit{Linear maps on $\star$-algebras acting on orthogonal elements like derivations or anti-derivations}, Filomat, 32(13) (2018), 4543--4554.
 
 \bibitem{Silver} 
S. Ponnusamy and H .Silverman, Complex Variables with Applications, Birkhauser, Boston, 2006. 
 
 \bibitem{wa}
L. Wang, Y. Fan, X. Ma, On bilinear maps determined by rank one matrices with some applications. Linear Algebra Appl. 434 (2011) , 1354--1361.

\end{thebibliography}

\end{document}